\documentclass{svmult}
\usepackage{amsmath}
\usepackage{amsfonts}
\usepackage{amssymb}
\usepackage{mathptmx}
\usepackage{helvet}
\usepackage{courier}
\usepackage{makeidx}
\usepackage{graphicx}
\usepackage{multicol}
\usepackage{multirow}
\usepackage{footmisc}
\usepackage{subfigure}

\newcommand{\G}{\vec{G}}
\newcommand{\U}{\vec{U}}
\newcommand{\n}{\vec{n}}

\newcommand{\mf}[1]{#1}

\newcommand{\Ltwo}{\textbf{\emph{L}}^2(\Omega)}
\newcommand{\ltwo}{L^2(\Omega)}
\newcommand{\V}{\mathbf{{V}}}
\newcommand{\W}{\mathbf{{W}}}
\newcommand{\Uh}{\mathcal{U}^h}
\newcommand{\Wh}{\mathcal{W}^h}

\newcommand{\norm}[1]{\left\Vert #1 \right\Vert}
\renewcommand{\div}{\mathrm{div}}

\begin{document}

\title*{A staggered discontinuous Galerkin method for a class of nonlinear elliptic equations}

\author{Eric T. Chung, Ming Fai Lam and Chi Yeung Lam}
\institute{
Eric T. Chung \at Department of Mathematics, The Chinese University of Hong Kong, Hong Kong SAR, \email{tschung@math.cuhk.edu.hk}\\
Ming Fai Lam \at Department of Mathematics, The Chinese University of Hong Kong, Hong Kong SAR, \email{mflam@math.cuhk.edu.hk}\\
Chi Yeung Lam \at Department of Mathematics, The Chinese University of Hong Kong, Hong Kong SAR, \email{cylam@math.cuhk.edu.hk}
}

\maketitle

\abstract{ In this paper, we present a staggered discontinuous Galerkin (SDG) method for a class of nonlinear elliptic equations in two dimensions. The SDG methods
have some distinctive advantages, and
have been successfully applied to a wide range of problems including
Maxwell equations, acoustic wave equation,
elastodynamics and incompressible Navier-Stokes equations.
Among many advantages of the SDG methods, one can apply a local post-processing technique to the solution,
and obtain superconvergence.
We will analyze the stability of the method and derive a priori error estimates.
We solve the resulting nonlinear system using the Newton's method, and the numerical results confirm the theoretical rates of convergence and superconvergence.}


\section{Introduction}

Our aim of this paper is to develop a staggered discontinuous Galerkin (SDG) method for a class of nonlinear elliptic problems arising in,
for example, hyperpolarization effects in electrostatic analysis \cite{hu2012nonlinear},
nonlinear magnetic field problems \cite{heise1994analysis},
subsonic flow problems \cite{feistauer1986finite},
and heat conduction.

A detailed introduction to the SDG method is given by \cite{chung2009optimal,chung2006optimal}.
This class of methods has been successfully applied to a wide range of problems including
the Maxwell equation \cite{chung2011staggered,chung2013convergence},
acoustic wave equation \cite{chung2009optimal},
elastic equations \cite{chung2015staggered,lee2016analysis},
and incompressible Navier-Stokes equations \cite{cheung2015staggered}.
In these applications, the approximate solutions obtain some nice properties such as
energy conservation, low dispersion error and mass conservation.
Recently, a connection between the SDG method and the hybridizable discontinuous Galerkin (HDG) method is obtained \cite{chung2014staggered,chung2016staggered}. From this perspective, the SDG method acquires some new properties, such as postprocessing and superconvergence properties, from the HDG method \cite{cockburn2009superconvergent}.
We remark that numerical methods based on staggered meshes are important in many applications,
see \cite{virieux1986p,tavelli2014staggered}.

To begin with, we let $\Omega\subset\mathbf{R}^2$ be a bounded and simply connected domain with polygonal boundary $\Gamma$.
Also, we let the coefficient $\varrho : \mathbf{R}^2 \to \mathbf{R}$ be a $L^{\infty}$ function satisfying certain conditions (will be specified).
Then, for a given $f\in L^2(\Omega)$ we seek $u \in H^1_0(\Omega)$ such that
\begin{eqnarray}
- \div \left(\varrho(\nabla u(x))\nabla{u(x)}\right) = f(x) \text{ in } \, \Omega, \mbox{ and }
u(x)  =  0 \quad \text {on } \, \Gamma, \label{pblm}
\end{eqnarray}
where $\div$ is the usual divergence operator.

This paper is organized as follows. In Section 2,
we will construct the SDG method.
In Section 3, we will discuss the implementation of the scheme.
In Section 4, we will prove stability estimates and an a priori error estimate of our scheme.
Finally, in Section 5, we will numerically show the rate of convergence of our method.
Throughout this paper, we use $C$ to denote a generic positive constant, which is independent of the mesh size.

\section{The SDG formulation}
We introduce new variables, the gradient $\mf{\G := \nabla u}$ and the flux $\mf{\U := \rho(\G)\G}$.
Then the problem \eqref{pblm} can be recasted as the following problem in $\Omega$:
Find $(\U, \G, u)$ such that,
\begin{equation}
\begin{gathered}
\G = \nabla u,\quad \U = \rho(\G)\G, \quad -\div\U = f \quad \mbox{ in } \Omega, \\
u = 0 \quad \mbox{ on }\Gamma.
\end{gathered}
\end{equation}

Next we describe the staggered mesh.
Assume $\Omega$ is triangulated by a family of triangles with no hanging nodes, namely, the initial triangulation $\mathcal{T}_u$.
The triangles in $\mathcal{T}_u$ are called the \emph{first-type macro element}.
We denote the set of all edges and all interior edges of $\mathcal{T}_u$ by $\mathcal{F}_u$ and $\mathcal{F}^0_{u}$, respectively.
Then we choose an interior point $\nu$ in each first-type macro element.
We denote the first-type macro element corresponding to $\nu$ by $\mathcal{S}(\nu)$.
By connecting each of these interior points to the three vertices of the triangle,
we subdivide each triangle into three subtriangles.
We denote the triangulation containing all these subtriangles by $\mathcal{T}$ and assume it is shape-regular.
We denote the set of all new edges in this subdivision process by $\mathcal{F}_p$.
Also, we denote the set of all edges
and the set of all interior edges by $\mathcal{F} := \mathcal{F}_u \cup \mathcal{F}_p$
and $\mathcal{F}_0 := \mathcal{F}^0_{u} \cup \mathcal{F}_p$, respectively.
For each interior edge $e\in \mathcal{F}^0_{u}$, there are two triangles $\tau_1, \tau_2 \in \mathcal{T}$
such that $e = \tau_1 \cap \tau_2$.
We denote the union $\tau_1 \cup \tau_2$ by $\mathcal{R}(e)$.
Also, for each boundary edge $e$, we denote the only triangle having $e$ as an edge by $\mathcal{R}(e)$.
These elements $\mathcal{R}(e)$ are called the \emph{second-type macro element}.
In Fig.~\ref{fig1}, we illustrate two first-type marco elements and a second-type marco element obtained from the subdividing process on two neighboring initial triangles.

\begin{figure}[h]
\centering
\includegraphics[width=3in]{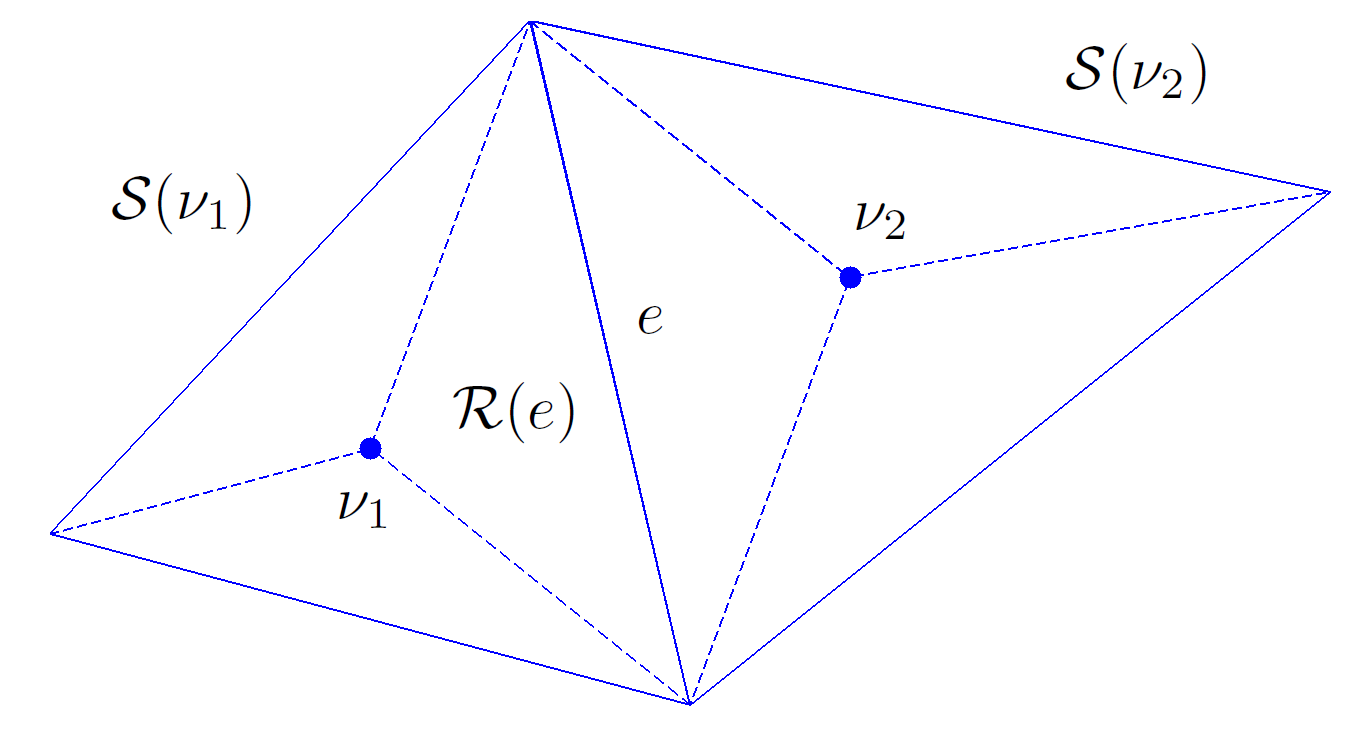}
\caption{An illustration of the triangulation $\mathcal{T}$.}\label{fig1}
\end{figure}

For a boundary edge $e$, we define $\n_e$ to be the unit normal vector pointing outside $\Omega$.
Otherwise, $\n_e$ is one of the two possible unit normal vectors of $e$.
When it is clear which edge is being considered, we will simply use $\n$ instead of $\n_e$.

Next, we describe the finite element spaces we use in our formulation.
Let $k \geq 0$ be a non-negative integer.
For each triangle $\tau  \in \mathcal{T}$,
we denote the space of polynomials on $\tau$ with degree at most $k$ by $P^k(\tau)$.
Then we define the \emph{{locally $H^1(\Omega)$-conforming finite element}} as
$$
\Uh:=\{v : v|_{\tau}\in P^k(\tau), \forall\tau\in\mathcal{T};v\text{ is continuous across}\;e\in\mathcal{F}^0_{u}\;;\;v|_{\partial\Omega}=0\},
$$
and the \textit{locally $H(\mathrm{div};\Omega)$-conforming finite element space} as
\begin{equation*}
\begin{gathered}
\Wh:=\{\V: \V|_{\tau}\in P^k(\tau)^2,\forall\tau\in\mathcal{T};  \text{the normal component }\V\cdot \n_e \\ \text{ across}\;e\in\mathcal{F}_p \text{ is continuous}\}.
\end{gathered}
\end{equation*}

Following \cite{chung2009optimal,chung2006optimal}, we consider the following formulation:
Find $(\U_h, \G_h, u_h) \in \Wh\times\Wh\times\Uh$
such that for any first-type element $\mathcal{S}(\nu)$ and second-type element $\mathcal{R}(e)$, any $(\V_h, \W_h, v_h) \in \Wh\times\Wh\times\Uh$,
\begin{equation}
\begin{gathered}
\int_{\mathcal{S}(\nu)} \G_h \cdot \V_h\, dx + \int_{\mathcal{S}(\nu)} u_h\, \div_h \V_h\, dx - \int_{\partial \mathcal{S}(\nu)} \G_h \left(\V_h\cdot \n\right) \, ds = 0, \\
\int_{\mathcal{S}(\nu)} \U_h \cdot \W_h\, dx - \int_{\mathcal{S}(\nu)}\rho(\G_h) \G_h \cdot \W_h\, dx = 0, \\
\int_{\mathcal{R}(e)} \U_h \cdot \nabla_h v_h \, dx - \int_{\partial \mathcal{R}(e)} (\U_h \cdot \n) v_h \, ds = \int_{\mathcal{R}(e)} f\cdot v_h\, \mf{dx},
\end{gathered}\label{sdgform}
\end{equation}
where $\nabla_h$ and $\div_h$ are the elementwise gradient and divergence operators, respectively. Besides, $\n$ denotes outward normals on $\mathcal{S}(\nu)$ or $\mathcal{R}(e)$ depending on the context.

We define the jump operator $[\cdot]$ as follows. For $e\in\mathcal{F}_{p}$, if $\tau_1, \tau_2 \in \mathcal{T}$
such that $e = \tau_1\cap \tau_2$ and $\n_e$ is pointing from $\tau_1$ to $\tau_2$,
then $$[v] := v|_{\tau_1} - v|_{\tau_2}.$$
For $e\in\mathcal{F}^0_{u}$, if $\tau_1, \tau_2 \in \mathcal{T}$
such that $e = \tau_1\cap \tau_2$ and $\n_e$ is pointing from $\tau_1$ to $\tau_2$,
then $$[\V\cdot \n_e] := \V|_{\tau_1}\cdot \n_e - \V|_{\tau_2}\cdot \n_e.$$
We also introduce two bilinear forms,
\begin{eqnarray*}
b_h(\V,v) := \int_{\Omega} \textbf{\emph{V}}\cdot \nabla_{\tau} v\;dx-\sum_{e \in \mathcal{F}_p}\int_e \textbf{\emph{V}} \cdot \textbf{\emph{n}}[v]\;d\sigma, \\
b_h^{\ast}(v,\V) := -\int_{\Omega} v\nabla_{\tau} \cdot \textbf{\emph{V}} \;dx+\sum_{e \in \mathcal{F}_u^0}\int_e v[\textbf{\emph{V}} \cdot \textbf{\emph{n}}]\;d\sigma,
\end{eqnarray*}
for $v \in \Uh, \textbf{\emph{V}} \in \Wh.$
Summing the equations in \eqref{sdgform} on $\mathcal{S}(\nu)$ and $\mathcal{R}(e)$, respectively,
we can recast \eqref{sdgform} into:
find $(\U_h, \G_h, u_h) \in \Wh\times\Wh\times\Uh$ such that,
\begin{equation}
	 \int_{\Omega} \G_h \cdot \V_h\, dx -b^{\ast}_h(u_h,\V_h)  =  0, \quad \text{ for any } \V_h \in \Wh; \label{sdg1}
\end{equation}
\begin{equation}
	\int_{\Omega} \U_h \cdot \W_h\,dx - \int_{\Omega} \rho(\G_h)\G_h \cdot \W_h\, dx  =  0, \quad \text{ for any } \W_h \in \Wh;\label{sdg2}
\end{equation}
\begin{equation}
	b_h(\U_h,v_h) =  \int_{\Omega} f v_h\, dx, \quad \text{ for any } v_h \in \Uh.\label{sdg3}
\end{equation}
This completes the definition of our SDG method.

\section{Implementation}
In this section we will discuss the implementation detail of our SDG method.
First of all we fix a basis $\lbrace \phi_i \rbrace_{i=1}^{N_u}$ for $\Uh$ and
$\lbrace \psi_i \rbrace_{\mf{i}=1}^{N_w}$ for $\Wh$, and write $u_h = \sum_i (\widehat{u}_h)_i \phi_i$,
$\G_h = \sum_i (\widehat{\G}_h)_i \psi_i$
and $\U_h = \sum_i (\widehat{\U}_h)_i \psi_i$,
where $\widehat{{u}}_h$, $\widehat{\G}_h$ and $\widehat{\U}_h$ are $N_u\times 1$, $N_w \times 1$ and $N_w \times 1$ vectors respectively.
Next, we define the mass matrix $M_h$ and the matrix $B_h$ by
$
(M_h)_{ij}  :=  \int_{\Omega} \psi_j \cdot \psi_i \, dx,
(B_h)_{ij}  :=  b_h(\psi_j, \phi_i), \mbox{ and}
(f_h)_{i}   :=  \int f v_i\, dx,
$
respectively.
Then we rewrite \eqref{sdg1}--\eqref{sdg3} as the following system:
\begin{equation}
M_h\widehat{\G}_h-B^T_h \widehat{u}_h  =  0, \label{system1}
\end{equation}
\begin{equation}
M_h\widehat{\U}_h  =  F(\widehat{\G}_h), \label{system2}
\end{equation}
\begin{equation}
B_h\widehat{\U}_h  =  \widehat{f}_h, \label{system3}
\end{equation}
where $F(\widehat{\G}_h)$ is a $N_w \times 1$ vector given by
$F(\widehat{\G}_h)_i:=\left(\rho({\G}_h){\G}_h,\psi_i\right)_{L^2(\Omega)}$.
Eliminating $\widehat{\textbf{\emph{U}}}_h$ from \eqref{system1}-\eqref{system3}, we obtain
\begin{eqnarray}\label{nsystem}
M_h\widehat{\G}_h-B^T_h \widehat{u}_h & = & 0,\\
B_h M_h^{-1}F(\widehat{\G}_h) & = & \widehat{f}_h.
\end{eqnarray}
Here $F$ is not a linear function in general. Hence we use Newton's method to solve this system.
Write $\widehat{\mathbf{x}}_h := ( \widehat{\G}_h, \widehat{u}_h )^T$ and
$H(\widehat{\mathbf{ x}}_h) := \left( M_h \widehat{\G}_h-B_h^T \widehat{u}_h,  B_h M_h^{-1}F(\widehat{\G}_h)-\mf{\widehat{f}}_h \right)^T$.
%
%
%
The Jacobian matrix of $H$ is given by
\begin{equation*}
  \begin{split}
    J(\mf{\widehat{\bf x}}_h):=\left(
    \begin{matrix}
         M_h & -B_h^T \\
         B_hM_h^{-1}F'(\widehat{\G}_h) & 0
    \end{matrix}\right)
  \end{split},
\end{equation*}
where $F'(\widehat{\G}_h)$ is the derivative with respect to $\widehat{\G}_h$, which is given by
\begin{equation*}
    F'(\widehat{\G}_h)_{ij}= \left(\rho(\G_h)\psi_j,\psi_i \right)+ \left((\nabla\rho(\G_h)\cdot\psi_j)\G_h,\psi_i \right).
\end{equation*}
Given an initial guess $\mf{\widehat{\mathbf{x}}}_h^0$, we repeatedly update $\mf{\widehat{\mathbf{x}}}_h^n$ by
\begin{equation}
\widehat{{\bf x}}_h^{n+1}= \widehat{{\bf x}}_h^n-[J(\widehat{{\bf x}}_h^n)]^{-1}H({\widehat{\bf x}}_h^n),
\end{equation}
until the successive error $\norm{u_h^{n+1} - u_h^{n}}$ is less than a given tolerance $\delta$.

\section{Stability and convergence of the SDG method}
We begin with some results from the SDG method for other problems.
We define the discrete $L^2$-norm $\Vert \cdot\Vert _X$ and the discrete $H^1$-norm $\Vert \cdot\Vert _Z$ for any $v\in \Uh$ by
\begin{equation}
  \Vert v\Vert ^2_X := \int_\Omega v^2\;dx+\sum_{e\in\mathcal{F}^0_u}h_e\int_e v^2\;d\sigma \quad \mbox{ and }
\end{equation}
\begin{equation}
  \Vert v\Vert ^2_Z := \int_\Omega |\nabla_h v|^2\;dx+\sum_{e\in\mathcal{F}_p}h^{-1}_e\int_e [v]^2\;d\sigma,
\end{equation}
respectively.
We also define the discrete discrete $L^2$-norm $\Vert \cdot\Vert _{X'}$
and the discrete $H^1$-norm $\Vert \cdot\Vert _{Z'}$ for any $\V \in \Wh$ by
\begin{equation}
    \Vert \V\Vert ^2_{X'} = \int_\Omega |\V|^2\;dx+\sum_{e\in\mathcal{F}_p}h_e\int_e(\V\cdot \textbf{\emph{n}})^2\;d\sigma,
\end{equation}
\begin{equation}
    \Vert \V\Vert ^2_{Z'} = \int_\Omega (\nabla\cdot \V)^2\;dx+\sum_{e\in\mathcal{F}^0_u}h^{-1}_e\int_e[\V\cdot \textbf{\emph{n}}]^2\;d\sigma.
\end{equation}

Then we recall some nice properties of the bilinear forms $b_h$ and $b_h^*$ introduced in previous section.
According to Lemma 2.4 of \cite{chung2009optimal},
\begin{equation} \label{beq}
  b_h(\V,v) = b_h^{\ast}(v,\V),\quad\forall(v,\V)\in\Uh\times\Wh.
\end{equation}
and the following inequality holds:
\begin{equation}\label{bineq}
  b_h(\V,v) \leq \Vert v\Vert _Z\Vert \V\Vert _{X'} ,\quad\forall(v,\V)\in\Uh\times\Wh.
\end{equation}
From the definition of $\norm{\cdot}_X$ and $\norm{\cdot}_{X'}$ it is clear that for any $v\in\Uh$ and $\V\in\Wh$,
\begin{eqnarray}
\norm{v}_{L^2(\Omega)} \leq \norm{v}_X, \label{L2leqX}\\
\norm{\V}_{L^2(\Omega)^2} \leq \norm{\V}_{X'} \label{L2leqX'}.
\end{eqnarray}
Using the argument in the proof of Lemma 2.1 in Arnold~\cite{arnold1982interior}, we have the following discrete Poincar\'e inequality.
\begin{lemma}\label{L2leqZnorm}
For any $v\in\Uh$, there is a positive constant $C$ independent of the mesh size $h$ such that
\begin{equation}
  \Vert v\Vert _{L^2(\Omega)}\leq C\Vert v\Vert_{Z}.
  \end{equation}
\end{lemma}
Moreover, the following inf-sup conditions holds for the bilinear forms $b_h$ and $b_h^*$.
\begin{lemma}\label{binfsup}
There is a constant $C$ independent of meshsize $h$ such that
\begin{eqnarray*}
    \inf_{\V\in\Wh}\sup_{v\in\Uh}\frac{b_h^{\ast}(v,\V)}{\Vert v\Vert _X\Vert \V\Vert _{Z'}}\geq C, \\
    \inf_{v\in\Uh}\sup_{\V\in\Wh}\frac{b_h(\V,v)}{\Vert v\Vert _{Z}\Vert \V\Vert _{X'}}\geq C.
\end{eqnarray*}
\end{lemma}

Next, we impose some \mf{restrictions} on the coefficient $\rho$.
We assume $\rho$ is bounded below by a positive number $\rho_0$.
Moreover, we follow Bustinza and Gatica~\cite{bustinza2004local}
to require $\rho(\W)\W$ to be \emph{strongly monotone}. In order words, there is a positive constant $C$ independent of $\V,\W\in\Ltwo^2$ such that
\begin{eqnarray}
\int_{\Omega}[\rho(\W)\W-\rho(\V)\V]\cdot(\W-\V)\geq C \left\Vert \W-\V \right\Vert _{L^2(\Omega)^2}^2; \label{strongmono}
\end{eqnarray}
We also require $\rho(\W)\W$ to be \emph{Lipschitz continuous}. In order words, there is a positive constant $C$ independent of $\V,\W\in\Ltwo^2$ such that
\begin{eqnarray}
\left\Vert\rho(\W)\W-\rho(\V)\V\right\Vert _{L^2(\Omega)^2}^2\leq C\left\Vert \W-\V \right\Vert_{L^2(\Omega)^2}^2. \label{lipschitz}
\end{eqnarray}

We will also consider the interpolants
$\mathcal{I}:H^1(\Omega)\rightarrow\Uh$ and $\mathcal{J}:H(\text{div};\Omega)\rightarrow\Wh$
discussed in \cite{chung2009optimal}, which is characterized by
\begin{eqnarray}
	b^{\ast}_h(\mathcal{I}u-u,\V) &=& 0, \quad\quad\forall u\in H^1(\Omega), \V \in \Wh, \label{bintepI}\\
	b_h(\mathcal{J}\U-\U,v) &=& 0, \quad\quad\forall \U\in H(\text{div};\Omega), v \in \Uh.
\end{eqnarray}
It is shown that
for any $v\in H^{k+1}(\Omega)$ and $\V\in H^{k+1}(\Omega)^2$, we have
    \begin{eqnarray}\label{Ierror}
      \Vert v-\mathcal{I}v\Vert _{\ltwo}\leq Ch^{k+1}\Vert v\Vert _{H^{k+1}(\Omega)}, \\
      \label{Jerror}
      \Vert \V-\mathcal{J}\V\Vert _{\Ltwo}\leq Ch^{k+1}\Vert \V\Vert _{H^{k+1}(\Omega)^2}.
    \end{eqnarray}

%
%
%
\begin{theorem}
    Let $(u,\G,\U)\in H^{k+1}(\Omega)\times H^{k+1}(\Omega)^2\times H^{k+1}(\Omega)^2$ be the solution of the original problem and $(u_h,\G_h,\U_h)$ be the solution of the SDG scheme \eqref{sdg1}--\eqref{sdg3}. Then we have the stability estimate
\begin{eqnarray}\label{stability}
      \Vert u_h\Vert _{L^2(\Omega)} + \norm{\mf{\U_h}}_{L^2(\Omega)^2} + \norm{ \mf{\G_h} }_{L^2(\Omega)^2} \leq C\Vert f\Vert _{L^2(\Omega)},
\end{eqnarray}
and the convergence estimates
\begin{eqnarray}
    \begin{split}
      \norm{u-u_h}_{L^2(\Omega)} + \norm{\U-\U_h}_{L^2(\Omega)^2} +\norm{ \G - \G_h }_{L^2(\Omega)^2} \\\leq Ch^{k+1} \left( \norm{u}_{H^{k+1}(\Omega)}+ \norm{\G}_{H^{k+1}(\Omega)^2} \right).
    \end{split}
\end{eqnarray}
\end{theorem}

\begin{proof}
We start by showing the stability estimate.
Taking $\mf{\W_h} = \G_h$, $\mf{\V_h}=\textbf{\emph{U}}_h$, $\mf{v_h}=u_h$ in \eqref{sdg1}--\eqref{sdg3},
summing and applying \mf{\eqref{beq}}, we have
\begin{equation}
	\int_{\Omega} \rho(\G_h)\G_h \cdot \G_h dx = \int_{\Omega} f u_h \, dx.
\end{equation}
%
%
Applying the Cauchy Schwarz inequality,
\begin{equation}
  \Vert \G_h\Vert _{L^2(\Omega)^2}^2
  \leq\rho_0^{-1}\Vert f\Vert _{L^2(\Omega)}\Vert u_h\Vert _{L^2(\Omega)}.  \label{stab0}
\end{equation}
Besides, using Lemma \ref{L2leqZnorm} and Lemma \ref{binfsup},
\begin{eqnarray}
	\Vert u_h\Vert _{L^2(\Omega)}
		\leq C\sup_{\V\in\Wh}\frac{b_h(\V,u_h)}{\Vert \V\Vert _{X'}}.\label{stab1}
\end{eqnarray}
Besides, using equations \eqref{beq} and \mf{\eqref{sdg1}}, we have for any $\V \in \Wh$
\begin{eqnarray}
	b_h(\V,u_h) = \int_{\Omega} \G_h \cdot \V\, dx \leq \norm{\G_h}_{L^2(\Omega)^2} \norm{\V}_{L^2(\Omega)^2}. \label{stab2}
\end{eqnarray}
Combining \eqref{stab1} and \eqref{stab2} and applying \eqref{L2leqX'},
\begin{eqnarray}
	\Vert u_h\Vert _{L^2(\Omega)}&\leq C\Vert \G_h\Vert_{L^2(\Omega)}.
\end{eqnarray}
Combining this with \eqref{stab0},
\begin{eqnarray}
	\Vert \G_h\Vert _{L^2(\Omega)^2}&\leq C\Vert f\Vert_{L^2(\Omega)},
\end{eqnarray}
and the stability estimate \eqref{stability} follows from the Lipschitz continuity \eqref{lipschitz}.

Next, we show the convergence of $\G$.
Note that \eqref{sdg1} and \eqref{sdg3} still holds if we replace $\G_h$ by $\G$, $\U_h$ by $\U$ and $u_h$ by $u$.
Therefore,
\begin{eqnarray}
        \int_{\Omega}(\G-\G_h)\cdot \V \, dx -b^{\ast}_h(u-u_h,\V)&=0\quad\forall \V\in\Wh, \label{conv0}\\
        b_h(\U-\U_h,v)&=0\quad\forall v\in\Uh.
\end{eqnarray}
Using the properties of $\mathcal{I}$ and $\mathcal{J}$,
\begin{eqnarray}
        \int_{\Omega}(\G-\G_h)\cdot \V\, dx-b^{\ast}_h(\mathcal{I}u-u_h,\V)&=&0\quad\forall \V\in\Wh,\\
        b_h(\mathcal{J}\U-\U_h,v)&=&0\quad\forall v\in\Uh.
\end{eqnarray}
In particular for $v=\mathcal{I}u-u_h$ and $\V=\mathcal{J}\U-\U_h,$ adding these two equations gives
\begin{eqnarray}
	\int_{\Omega}(\G-\G_h)\cdot(\mathcal{J}\U-\U_h)\, dx = 0. \label{conv1}
\end{eqnarray}
On the other hand, from the strong monotonicity \eqref{strongmono},
\begin{eqnarray}
\norm{\mathcal{J}\G-\G_h}^2_{L^2(\Omega)^2} \leq \int_{\Omega}(\mathcal{J}\G-\G_h)\cdot(\mathcal{J}\U-\U_h)\, dx.
\end{eqnarray}
Applying equation \eqref{conv1},
\begin{eqnarray}
\begin{split}
\norm{\mathcal{J}\G-\G_h}^2_{L^2(\Omega)^2} &\leq \int_{\Omega}(\mathcal{J}\G-G)\cdot(\mathcal{J}\U-\U_h)\, dx \\
&\leq C \norm{\mathcal{J}\G-\G}_{L^2(\Omega)^2} \norm{\mathcal{J}\U-\U_h}_{L^2(\Omega)^2}.
\end{split}
\end{eqnarray}
Applying the Lipschitz continuity \eqref{lipschitz},
\begin{eqnarray}
\norm{\mathcal{J}\G-\G_h}_{L^2(\Omega)^2}
&\leq C\norm{\mathcal{J}\G-\G}_{L^2(\Omega)^2}.
\end{eqnarray}
Hence applying \eqref{Jerror},
\begin{equation}
    \begin{split}
      \norm{\G-\G_h}_{L^2(\Omega)^2}&\leq \norm{\G-\mathcal{J}\G}_{L^2(\Omega)^2}+\norm{\mathcal{J}\G-\G_h}_{L^2(\Omega)^2}\\
      &\leq Ch^{k+1}\norm{\G}_{H^{k+1}(\Omega)^2}.
      \end{split}
\end{equation}

Then we show the convergence of $u$. Using equation (\ref{Ierror}),
\begin{equation}
    \begin{split}
      \Vert u-u_h\Vert _{L^2(\Omega)}&\leq \Vert u-\mathcal{I}u\Vert _{L^2(\Omega)}+\Vert \mathcal{I}u-u_h\Vert _{L^2(\Omega)}\\
      &\leq Ch^{k+1}\Vert u\Vert _{H^{k+1}(\Omega)}+\Vert \mathcal{I}u-u_h\Vert _{L^2(\Omega)}.
      \end{split}
\end{equation}
Using the inf-sup condition in Lemma \ref{binfsup}, equation \eqref{beq}, \eqref{bintepI}, \eqref{conv0} and \eqref{L2leqX'},
\begin{equation}
      \begin{split}
        \Vert \mathcal{I}u-u_h\Vert _{L^2(\Omega)}&\leq C\sup_{\V\in\Wh} \frac{b_h(\V,\mathcal{I}u-u_h)}{\Vert \V\Vert _{X'}}
                                                          = C\sup_{\V\in\Wh} \frac{b_h^{\ast}(\mathcal{I}u-u_h,\V)}{\Vert \V\Vert _{X'}}\\
                                                          &= C\sup_{\V\in\Wh} \frac{b_h^{\ast}(u-u_h,\V)}{\Vert \V\Vert _{X'}}
                                                          = C\sup_{\V\in\Wh} \frac{\int_{\Omega}(\G-\G_h)\cdot\V\, dx}{\Vert \V\Vert _{X'}}\\
                                                          &\leq C\norm{\G-\G_h}_{L^2(\Omega)^2},
      \end{split}
\end{equation}
which shows the convergence of $u_h$. The convergence of $\U_h$ follows from the Lipschitz continuity \eqref{lipschitz}.
\end{proof}

\section{Numerical examples}
In this section, we present some numerical examples and verify the convergence rate of our SDG method.
Moreover, we will obtain a postprocessed solution $u_h^*$ which converges with higher order than $u_h$.
We define the postprocessed solution $u_h^*$ as follows.
For each $\tau\in\mathcal{T}$, we take $u_h^* \in P^{k+1}(\tau)$ determined by
\begin{equation}
\int_{\tau} \nabla u^*_h \cdot \nabla w \, dx  =  \int_{\tau} \mf{\G_h} \cdot \nabla w\, dx, \quad \forall w \in P^{k+1}(\tau)^0
\end{equation}
\begin{equation}
\int_{\tau} u_h^* \, dx  =  \int_{\tau} u_h \, dx,
\end{equation}
	where $P^{k+1}(\tau)^0 := \left\lbrace w\in P^{k+1}(\tau): \int_{\tau} w\, dx = 0 \right\rbrace$. See \cite{cockburn2009superconvergent}.
	
For all of our numerical examples, We consider square domain $\Omega = [0,1]^2$.
We divide this domain into $N \times N$ squares
and divide each square into two triangles.
We use this as our initial triangulation $\mathcal{T}_u$
and further subdivide each triangle taking the interior points as the centroids of the triangles following the discussion in Section 2. We take the mesh size $h:=1/N$. We illustrate the mesh with $h=1/4$ in Fig.~\ref{mesh}.
\begin{figure}[h]
\centering
\includegraphics[width=2in]{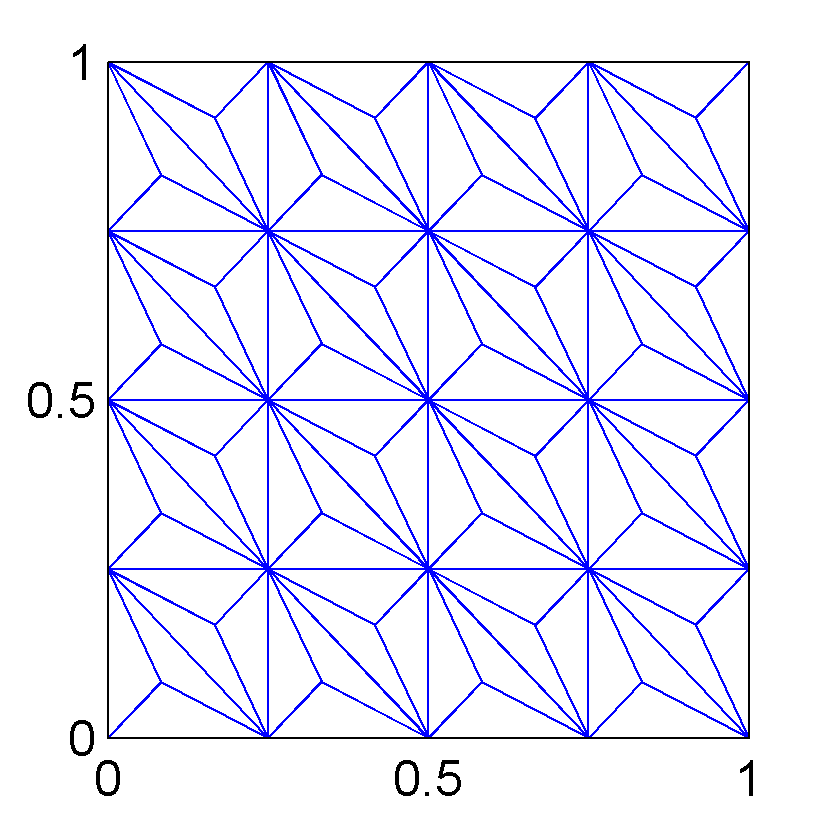}%
\caption{Triangulation on $\Omega = [0,1]^2$ with mesh size 1/4.}
\label{mesh}
\end{figure}
We consider the following solutions of equation \eqref{pblm}.
\begin{equation*}
  \begin{split}
    u_1(x,y) &= \sin(\pi x)\sin(\pi y),\\
    u_2(x,y) &= 10xy^2(1-x)(1-y)-\frac{\E^{x-1}\sin(\pi x)\sin(\pi y)}{2}.\\
  \end{split}
\end{equation*}
All these solutions have zero value on the boundary of $\Omega$.
We also consider the following six nonlinear coefficients to test the order of convergence.
\begin{align*}
\rho_1(\nabla u) & :=  2+\dfrac{1}{1+|\nabla u|} &
\rho_2(\nabla u) & :=  1+\exp(-|\nabla u|) \\
\rho_3(\nabla u) & :=  1+\exp(-|\nabla u|^2) &
\rho_4(\nabla u) & :=  \dfrac{1}{\sqrt{1+|\nabla u|}} \\
\rho_5(\nabla u) & :=  |\nabla u| &
\rho_6(\nabla u) & :=  |\nabla u|^2
\end{align*}

For each $u_j$ and $\rho_{\ell}$, we choose $f$ in \eqref{pblm} and
solve for the approximate solution in the spaces of piecewise linear polynomial (i.e. $k=1$), using Newton's iteration.
We terminate the Newton's iteration when the successive error is less than $\delta = 10^{-10}$.
Let $u_{j,h}$ be the approximate solution obtained from this Newton's iteration, and $u_{j,h}^*$ be the solution obtained from applying the above postprocessing procedure to $u_{j,h}$.
Under different choice of nonlinear coefficients $\rho_{\ell}$, we compute the $L^2$ error for $u_{j,h}$ and $u_{j,h}^*$, given by
$\Vert u_j-u_{j,h}\Vert _{\ltwo}$ and $\Vert u_j-u^*_{j,h}\Vert _{\ltwo}$, respectively. The results are listed in Table \ref{table_u1}
and Table \ref{table_u2}.
From these results, we see clearly that the scheme gives optimal rate of convergence for the numerical solution
and superconvergence for the postprocessed solution.

\begin{table}[ht]
\begin{tabular}{ccccccc}
Coefficient & Mesh size  & $||u_1-\mf{u_{1,h}}||_{L^2(\Omega)}$ & order & $||u_1-\mf{u_{1,h}^\ast}||_{L^2(\Omega)}$ & order & Number of iterations \\ \hline
\multirow{5}{*}{$\rho_1$}& $1/4$ & 3.54e-2 & - & 2.86e-3 & - & 4 \\
& $1/8$ & 9.24e-3 & 1.94 & 3.71e-4 & 2.95 & 4 \\
& $1/16$ & 2.34e-3 & 1.98 & 4.70e-5 & 2.98 & 4 \\
& $1/32$ & 5.86e-4 & 2.00 & 5.91e-6 & 2.99 & 4 \\
& $1/64$ & 1.46e-4 & 2.00 & 7.40e-7 & 3.00 & 4 \\ \hline
\multirow{5}{*}{$\rho_2$}& $1/4$ & 3.50e-2 & - & 3.00e-3 & - & 4 \\
& $1/8$ & 9.23e-3 & 1.92 & 3.95e-4 & 2.82 & 4 \\
& $1/16$ & 2.34e-3 & 1.98 & 5.07e-5 & 2.96 & 4 \\
& $1/32$ & 5.86e-4 & 2.00 & 6.45e-6 & 2.98 & 4 \\
& $1/64$ & 1.46e-4 & 2.00 & 8.13e-7 & 2.99 & 4 \\ \hline
\multirow{5}{*}{$\rho_3$}& $1/4$ & 3.78e-2 & - & 4.31e-3 & - & 5 \\
& $1/8$ & 9.41e-3 & 2.01 & 5.46e-4 & 2.98 & 5 \\
& $1/16$ & 2.34e-3 & 2.01 & 5.81e-5 & 3.23 & 5 \\
& $1/32$ & 5.86e-4 & 2.00 & 7.67e-6 & 2.92 & 5 \\
& $1/64$ & 1.46e-4 & 2.00 & 9.84e-7 & 2.96 & 5 \\ \hline
\multirow{5}{*}{$\rho_4$} & $1/4$ & 3.50e-2 & - & 3.13e-3 & - & 4 \\
& $1/8$ & 9.21e-3 & 1.93 & 4.12e-4 & 2.93 & 5 \\
& $1/16$ & 2.34e-3 & 1.98 & 5.30e-5 & 2.96 & 5 \\
& $1/32$ & 5.86e-4 & 2.00 & 6.74e-6 & 2.98 & 5 \\
& $1/64$ & 1.46e-4 & 2.00 & 8.49e-7 & 2.99 & 5 \\ \hline
\multirow{5}{*}{$\rho_5$} & $1/4$ & 3.60e-2 & - & 3.32e-3 & - & 6 \\
& $1/8$ & 9.29e-3 & 1.95 & 5.42e-4 & 2.62 & 6 \\
& $1/16$ & 2.34e-3 & 1.99 & 8.34e-5 & 2.70 & 7 \\
& $1/32$ & 5.86e-4 & 2.00 & 1.20e-6 & 2.79 & 8 \\
& $1/64$ & 1.47e-4 & 2.00 & 1.67e-7 & 2.85 & 8 \\ \hline
\multirow{5}{*}{$\rho_6$} & $1/4$ & 3.56e-2 & - & 5.98e-3 & - & 10 \\
& $1/8$ & 9.29e-3 & 1.94 & 1.50e-3 & 2.00 & 10 \\
& $1/16$ & 2.34e-3 & 1.99 & 2.28e-4 & 2.72 & 14 \\
& $1/32$ & 5.86e-4 & 2.00 & 3.18e-5 & 2.84 & 20 \\
& $1/64$ & 1.47e-4 & 2.00 & 4.29e-6 & 2.89 & 27 \\ \hline
\end{tabular}
\caption{The $L^2$ error for $\mf{u_{1,h}}$ and $\mf{u_{1,h}^*}$ under difference choices of coefficients.}
\label{table_u1}
\end{table}

\begin{table}[ht]
\begin{tabular}{ccccccc}
Coefficient & Mesh size  & $||u_2-\mf{u_{2,h}}||_{L^2(\Omega)}$ & order & $||u_2-\mf{u_{2,h}^\ast}||_{L^2(\Omega)}$ & order & Number of iterations \\ \hline
\multirow{5}{*}{$\rho_1$} & $1/4$ & 1.46e-2 & - & 1.78e-3 & - & 5 \\
& $1/8$ & 3.91e-3 & 1.90 & 2.40e-4 & 2.88 & 5 \\
& $1/16$ & 9.92e-4 & 1.98 & 3.11e-5 & 2.95 & 5 \\
& $1/32$ & 2.49e-4 & 1.99 & 3.94e-6 & 2.98 & 5 \\
& $1/64$ & 6.24e-5 & 2.00 & 5.00e-7 & 2.98 & 5 \\ \hline
\multirow{5}{*}{$\rho_2$} & $1/4$ & 1.45e-2 & - & 1.72e-3 & - & 5 \\
& $1/8$ & 3.90e-3 & 1.90 & 2.32e-4 & 2.89 & 5 \\
& $1/16$ & 9.91e-4 & 1.98 & 3.04e-5 & 2.93 & 5 \\
& $1/32$ & 2.45e-4 & 1.99 & 3.82e-6 & 2.99 & 5 \\
& $1/64$ & 6.24e-5 & 2.00 & 4.94e-7 & 2.95 & 5 \\ \hline
\multirow{5}{*}{$\rho_3$} & $1/4$ & 1.40e-2 & - & 1.90e-3 & - & 6 \\
& $1/8$ & 3.94e-3 & 1.83 & 2.58e-4 & 2.88 & 6 \\
& $1/16$ & 9.94e-4 & 1.99 & 3.22e-5 & 3.00 & 6 \\
& $1/32$ & 2.49e-4 & 1.99 & 4.19e-6 & 2.94 & 6 \\
& $1/64$ & 6.24e-5 & 2.00 & 5.33e-7 & 2.97 & 6 \\ \hline
\multirow{5}{*}{$\rho_4$} & $1/4$ & 1.45e-2 & - & 1.71e-3 & - & 4 \\
& $1/8$ & 3.90e-3 & 1.89 & 2.31e-4 & 2.89 & 4 \\
& $1/16$ & 9.91e-4 & 1.98 & 3.03e-5 & 2.93 & 4 \\
& $1/32$ & 2.49e-4 & 1.99 & 3.79e-6 & 3.00 & 4 \\
& $1/64$ & 6.24e-5 & 2.00 & 4.89e-7 & 2.95 & 4 \\ \hline
\multirow{5}{*}{$\rho_5$} & $1/4$ & 1.49e-2 & - & 3.97e-3 & - & 7 \\
& $1/8$ & 3.94e-3 & 1.92 & 6.05e-4 & 2.71 & 8 \\
& $1/16$ & 9.99e-4 & 1.98 & 9.24e-5 & 2.71 & 8 \\
& $1/32$ & 2.50e-4 & 2.00 & 1.32e-5 & 2.81 & 10 \\
& $1/64$ & 6.25e-5 & 2.00 & 1.79e-6 & 2.88 & 10 \\ \hline
\multirow{5}{*}{$\rho_6$} & $1/4$ & 1.54e-2 & - & 6.54e-3 & - & 13 \\
& $1/8$ & 3.91e-3 & 1.98 & 1.17e-3 & 2.49 & 15 \\
& $1/16$ & 9.94e-4 & 1.98 & 1.96e-4 & 2.58 & 18 \\
& $1/32$ & 2.50e-4 & 1.99 & 2.92e-5 & 2.74 & 21 \\
& $1/64$ & 6.24e-5 & 2.00 & 3.91e-6 & 2.90 & 23 \\ \hline
\end{tabular}
\caption{The $L^2$ error for $u_{2,h}$ and $u_{2,h}^*$ under difference choices of coefficients.}
\label{table_u2}
\end{table}


\section*{Acknowledgement}

The research of Eric Chung is partially supported by Hong Kong RGC General Research Fund (Project: 14301314).

\bibliography{ref}{}

\begin{thebibliography}{10}

\bibitem{arnold1982interior}
D.~N. Arnold.
\newblock An interior penalty finite element method with discontinuous
  elements.
\newblock {\em SIAM journal on numerical analysis}, 19(4):742--760, 1982.

\bibitem{bustinza2004local}
R.~Bustinza and G.~N. Gatica.
\newblock A local discontinuous {Galerkin} method for nonlinear diffusion
  problems with mixed boundary conditions.
\newblock {\em SIAM Journal on Scientific Computing}, 26(1):152--177, 2004.

\bibitem{cheung2015staggered}
S.~W. Cheung, E.~Chung, H.~H. Kim, and Y.~Qian.
\newblock Staggered discontinuous {Galerkin} methods for the incompressible
  {Navier}--{Stokes} equations.
\newblock {\em Journal of Computational Physics}, 302:251--266, 2015.

\bibitem{chung2014staggered}
E.~Chung, B.~Cockburn, and G.~Fu.
\newblock The staggered {DG} method is the limit of a hybridizable {DG} method.
\newblock {\em SIAM Journal on Numerical Analysis}, 52(2):915--932, 2014.

\bibitem{chung2016staggered}
E.~Chung, B.~Cockburn, and G.~Fu.
\newblock The staggered {DG} method is the limit of a hybridizable {DG} method.
  {P}art {II}: the {S}tokes flow.
\newblock {\em Journal of Scientific Computing}, 66(2):870--887, 2016.

\bibitem{chung2013convergence}
E.~T. Chung, P.~Ciarlet, and T.~F. Yu.
\newblock Convergence and superconvergence of staggered discontinuous
  {G}alerkin methods for the three-dimensional {M}axwell?s equations on
  {C}artesian grids.
\newblock {\em Journal of Computational Physics}, 235:14--31, 2013.

\bibitem{chung2006optimal}
E.~T. Chung and B.~Engquist.
\newblock Optimal discontinuous {Galerkin} methods for wave propagation.
\newblock {\em SIAM Journal on Numerical Analysis}, 44(5):2131--2158, 2006.

\bibitem{chung2009optimal}
E.~T. Chung and B.~Engquist.
\newblock Optimal discontinuous {Galerkin} methods for the acoustic wave
  equation in higher dimensions.
\newblock {\em SIAM Journal on Numerical Analysis}, 47(5):3820--3848, 2009.

\bibitem{chung2015staggered}
E.~T. Chung, C.~Y. Lam, and J.~Qian.
\newblock A staggered discontinuous {Galerkin} method for the simulation of
  seismic waves with surface topography.
\newblock {\em Geophysics}, 80(4):T119--T135, 2015.

\bibitem{chung2011staggered}
E.~T. Chung and C.~S. Lee.
\newblock A staggered discontinuous {Galerkin} method for the curl--curl
  operator.
\newblock {\em IMA Journal of Numerical Analysis}, page drr039, 2011.

\bibitem{cockburn2009superconvergent}
B.~Cockburn, J.~Guzm{\'a}n, and H.~Wang.
\newblock Superconvergent discontinuous {G}alerkin methods for second-order
  elliptic problems.
\newblock {\em Mathematics of Computation}, 78(265):1--24, 2009.

\bibitem{feistauer1986finite}
M.~Feistauer.
\newblock On the finite element approximation of a cascade flow problem.
\newblock {\em Numerische Mathematik}, 50(6):655--684, 1986.

\bibitem{heise1994analysis}
B.~Heise.
\newblock Analysis of a fully discrete finite element method for a nonlinear
  magnetic field problem.
\newblock {\em SIAM Journal on Numerical Analysis}, 31(3):745--759, 1994.

\bibitem{hu2012nonlinear}
L.~Hu and G.-W. Wei.
\newblock Nonlinear poisson equation for heterogeneous media.
\newblock {\em Biophysical journal}, 103(4):758--766, 2012.

\bibitem{lee2016analysis}
J.~J. Lee and H.~H. Kim.
\newblock Analysis of a staggered discontinuous {G}alerkin method for linear
  elasticity.
\newblock {\em Journal of Scientific Computing}, 66(2):625--649, 2016.

\bibitem{tavelli2014staggered}
M.~Tavelli and M.~Dumbser.
\newblock A staggered semi-implicit discontinuous {G}alerkin method for the two
  dimensional incompressible {N}avier-{S}tokes equations.
\newblock {\em Applied Mathematics and Computation}, 248:70--92, 2014.

\bibitem{virieux1986p}
J.~Virieux.
\newblock {P}-{SV} wave propagation in heterogeneous media: {V}elocity-stress
  finite-difference method.
\newblock {\em Geophysics}, 51(4):889--901, 1986.

\end{thebibliography}
\bibliographystyle{abbrv}

\end{document}